\documentclass[reqno,11pt]{amsart} 

\usepackage{amsmath, amsthm, a4, latexsym, amssymb, pstricks, setspace, scrtime}
\usepackage{indentfirst, subfiles}
\usepackage{graphicx}
\usepackage{amsfonts,xcolor}
\usepackage{tikz}
\usepackage{pgfplots}
\usepackage{hyperref}
\usetikzlibrary{patterns}
\usetikzlibrary{plothandlers}
\usetikzlibrary{plotmarks}
\usepgflibrary{plothandlers}
\usepgflibrary{plotmarks}
\usetikzlibrary{decorations.pathreplacing}
\usepackage{etoolbox}
\usepackage[numbers]{natbib}
\usepackage{bbm}

\def\@rmrk#1#2{\refstepcounter
    {#1}\@ifnextchar[{\@yrmrk{#1}{#2}}{\@xrmrk{#1}{#2}}}

\makeatletter\@addtoreset{equation}{section}\makeatother

 \newfont{\bfit}{cmbxti10 scaled 1200}

 \newcommand{\eps}{\varepsilon}

 \newcommand{\R}{\mathbb{R}}
 \newcommand{\N}{\mathbb{N}}

 \newcommand{\1}{{\sf 1}}

 \newcommand{\heap}[2]{\genfrac{}{}{0pt}{}{#1}{#2}}
 \newcommand{\sfrac}[2]{\mbox{$\frac{#1}{#2}$}}
 
 \newcommand{\ssup}[1] {{\scriptscriptstyle{({#1}})}}

\renewcommand{\subsection}{\secdef \subsct\sbsect}
\newcommand{\subsct}[2][default]{\refstepcounter{subsection}
\vspace{0.15cm}
{\flushleft\bf \arabic{section}.\arabic{subsection}~\bf #1  }
\nopagebreak\nopagebreak}
\newcommand{\sbsect}[1]{\vspace{0.1cm}\noindent
{\bf #1}\vspace{0.1cm}}

{\nopagebreak {\hfill\rule{2mm}{2mm}}
\\ }

\newtheorem{theorem}{Theorem}[section]

\newtheorem{lemma}[theorem]{Lemma}

\newtheorem{prop}[theorem]{Proposition}

\renewcommand{\P}{\mathbb{P}}

\newcommand{\E}{\mathbb{E}}

\newcommand{\X}{\mathcal X}
\newcommand{\Y}{\mathcal Y}   
\newcommand{\setZ}{\mathcal Z}
\newcommand{\V}{\mathcal V}
\newcommand{\x}{\mathbf{x}}
\newcommand{\y}{\mathbf{y}}

\newcommand{\T}{\mathbb{T}}

\newcommand{\set}[1]{\left\{ #1 \right\}}

\newcommand{\abs}[1]{\left| #1 \right|}

\newtheoremstyle{thm}{1.5ex}{1.5ex}{\itshape\rmfamily}{}
{\bfseries\rmfamily}{}{2ex}{}

\newtheoremstyle{rem}{1.3ex}{1.3ex}{\rmfamily}{}
{\itshape\rmfamily}{}{1.5ex}{}
\theoremstyle{rem}

\refstepcounter{subsubsection}

\def\thebibliography#1{\section*{References}
  \list%
  {\arabic{enumi}.}
    {\settowidth\labelwidth{[#1]}\leftmargin\labelwidth
    \advance\leftmargin\labelsep
    \parsep0pt\itemsep0pt
    \usecounter{enumi}}
    \def\newblock{\hskip .11em plus .33em minus .07em}
    \sloppy                   
    \sfcode`\.=1000\relax}

\usepackage{cleveref} 

\usepackage{todonotes}

\begin{document}

\title[The age-dependent random connection model]
{The age-dependent random connection model}

\centerline{\LARGE \bf The age-dependent random connection model}

\vspace{0.4cm}

\thispagestyle{empty}
\vspace{0.3cm}
\textsc{Peter Gracar, Arne Grauer, Lukas L\"uchtrath and Peter M\"orters\footnote{Communicating author.}}\\
Mathematisches Institut, Universit\"at zu K\"oln, 50931 K\"oln, Germany\\
E--mail: \texttt{moerters@math.uni-koeln.de} 

\vspace{0.3cm}

\vspace{0.5cm}

\begin{quote}{\small {\bf Abstract: } We investigate a class of growing graphs embedded into the $d$-dimensional torus where new vertices arrive according to a Poisson process in time, are randomly placed in space and connect to existing vertices with a probability depending on time, their spatial distance and their relative ages. This  simple model for a scale-free network is called the \emph{age-based spatial preferential attachment network} and is based on the idea of preferential attachment with spatially induced clustering. We show that the graphs converge weakly locally to a variant of the random connection model, which we call the \emph{age-dependent random connection model}. This is a natural infinite graph on a Poisson point process where points are marked by a uniformly distributed age and connected with a probability depending on their spatial distance and both ages. We use the limiting structure to investigate asymptotic degree distribution, clustering coefficients and typical edge lengths in the age-based spatial preferential attachment network. }
\end{quote}
\vspace{0.4cm}

{\footnotesize
{\bf MSc classification (2010):} Primary 05C80 Secondary 60K35.\\[-5mm]

\mbox{geometric} random graphs, spatially embedded graphs, clustering coefficient, power-law degree distribution, edge lengths.}
\bigskip



\section{Motivation: Scale-free networks and clustering} \label{Section:Motivation}

Networks arising in different contexts, be it social, communication, technological or biological networks, often have strikingly similar features.  The question of interest to us is why this is the case. Can these features be explained by a few basic principles underpinning the construction of these networks?
\medskip

The probabilistic methodology to approach this question is to build a network model as a growing sequence of random graphs defined from very simple interaction principles and to prove the emerging features in the form of limit theorems when the  number of vertices is going to infinity. The purpose of this paper is to present a tractable network model which is based on simple and natural construction principles and for which such limit theorems can be proved rigorously. We will call this model the \emph{age-based preferential attachment model}.
\medskip

Potential features of networks we are interested in include:
\begin{itemize}
	\item Networks are \emph{scale-free}: 
	For very large network size and 
	large $k$ the proportion $\mu(k)$ of nodes with exactly $k$ neighbours is of order $k^{-\tau+o(1)}$
	for some \emph{power law exponent $\tau$}.\smallskip
	
	\item Networks are \emph{ultrasmall}: 
	The shortest path between two randomly chosen nodes in the graph is doubly logarithmic
	in the number of vertices. \smallskip
	
	\item Networks are \emph{robust} under random attack: 
	If an arbitrarily large 
	proportion of links is  randomly removed from the network the qualitative topological features of 
	the network remain unchanged.\smallskip
	
	
	\item Networks are \emph{vulnerable} under targeted attack: Even if only a small number of 
	the most influential nodes are removed, the topological features of the network change dramatically.
	\smallskip
	
	\item Networks show strong \emph{clustering}: Nodes picked from the neighbourhood of a typical
	node have a much higher chance of being connected by a link than randomly picked nodes.
	
	%
\end{itemize}
These features should {emerge} solely from the principles on which our model rests. In this paper the focus is on the scale-free and clustering properties of our model. We also believe that the other properties hold in certain parameter ranges, these (harder) properties are left for future research we will undertake.

The simple building principles for our network are:
\begin{itemize}
	\item { They are {built dynamically} by adding nodes successively.}
	\item { When a new node is introduced, it prefers to establish links to existing nodes that are either
		\begin{itemize} 
			\item {powerful} or old;
			\item or {similar} to the new node.
	\end{itemize}}
\end{itemize}
The idea of building a network by connecting incoming nodes to existing nodes with a probability depending increasingly on their power
was introduced into network theory by Barab\'{a}si and Albert~\cite{BA99}. They use the degree of the existing node as the indicator of its power;
we speak of \emph{degree-based preferential attachment}. There is now a substantial body of work showing rigorously that the resulting networks are 
scale-free~\cite{bollobas} and, for power-law exponent $\tau<3$, they are  ultrasmall and robust~\cite{Dommers2010, dereich2012, dereich2013}. The key technical tool in the proofs of the latter properties is the coupling of neighbourhoods of typical vertices to well-studied random tree models often coming from the genealogy of branching processes. This technique rests therefore crucially on the absence of clustering, as clustering leads to presence of short cycles that destroy the local tree structure of the graphs.
\medskip

To include clustering the idea of preferential attachment has to be developed further. An attractive approach is to include the idea that nodes have individual 
features and similarity of the features of nodes is a further incentive to form links between two nodes. This is realised by embedding the graphs into space 
and giving preference to short edges. We speak of \emph{spatially induced clustering}. 
Spatial preferential attachment models were studied by Manna and Sen~\cite{MS02}, Flaxman, Frieze and 
Vera~\cite{FlaxmanFriezeVera2006, FlaxmanFriezeVera2007},  Aiello, Bonato, Cooper, Janssen and 
Pra\l at~\cite{AielloEtAll2008}, Jordan~\cite{Jordan2010, Jordan2013}, Janssen, Pra\l at, Wilson~\cite{JanssenPralatWilson2013}, Jordan and Wade~\cite{JordanWade2015}, and Jacob and M\"orters~\cite{JacobMoerters2015, JacobMoerters2017}.
\medskip

The spatial preferential attachment models studied in these papers appear to be too complicated to fully characterise features like robustness or ultrasmallness. We therefore propose a simpler spatial model where preferential attachment is not to vertices with high degree but to vertices with old age; we speak of \emph{age-based preferential attachment}. Age-based models are easier to study because while the actual degree of a vertex depends in a complex way on the graph geometry, the age of a vertex is a given quantity. At the same time there is a strong link between degree and age as the expected degree is a simple function of the age of a vertex. Our simplification therefore removes complicated but (on a large scale) inessential correlations between edges and allows us to focus on the important correlations, namely those coming from the spatial embedding. 
\medskip

\pagebreak[3]

\section{The age-based spatial preferential attachment network} \label{Section:ABPAN}

The \emph{age-based spatial preferential attachment model} is a {growing sequence of graphs $(G_t)_{t>0}$ in continuous time}. The vertices of
the graphs are embedded in the $d$-dimensional torus $\mathbb{T}_1^d=(-1/2,1/2]^d$ of side-length one, endowed with the torus metric $d$ defined by
\begin{align*}
d(x,y)=\min\left\{\left| x-y+u\right| \colon u\in\{-1,0,1\}^d\right\} \mbox{for $x,y\in\mathbb{T}_1^d$,}
\end{align*}
where here and throughout the paper $\abs{\cdot}$ denotes the Euclidean norm.
Vertices are denoted by $\mathbf{y}=(y,s)$ and they are characterised by their birthtime $s>0$ and by their position $y\in \mathbb{T}_1^d$.
\smallskip

At time $t=0$ the graph $G_0$ has no vertices or edges. Then
\begin{itemize}
	\item Vertices arrive accroding to a standard Poisson process in time 
		and are placed independently uniformly on the $d$-dimensional torus~$\mathbb{T}_1^d$.
	\item Given the graph $G_{t-}$ a vertex  $\mathbf{x}=(x,t)$ born at time $t$ and placed in position $x$
		is connected by an edge to each existing node $\mathbf{y}=(y,s)$ {independently} with probability 
		\begin{align}\varphi\bigg(\frac{t \cdot d(x,y)^d}{\beta\cdot \big(\frac{t}s\big)^\gamma}\bigg),\label{constructionProb}
		\end{align}
\end{itemize}
	where 
\begin{enumerate}
	\item[(a)] $\varphi\colon[0,\infty)\rightarrow[0,1]$ is the \textit{profile function}. It is nonincreasing, integrable and normalized in the sense that
	\begin{align}
		\int\limits_{\mathbb{R}^d} \varphi(| x|^d)\, dx=1. \label{profileNormalization}
    \end{align} 
    The profile function can be used to control the occurrence of long edges.\smallskip
    \item[(b)]  $\gamma\in(0,1)$ is a parameter that quantifies the strength of the preferential attachment mechanism. We shall see that it alone determines the power-law exponent of the network.\smallskip
    \item[(c)] $\beta\in(0,\infty)$ is a parameter to control the edge density, which is asymptotically equal to $\frac{\beta}{1-\gamma}$, hence the smaller  $\beta$, the sparser the graph.
\end{enumerate}
\medskip

Some comments on our choices in \eqref{constructionProb} are in order.
\begin{enumerate}
	\item[(i)] For any $r>0$, the profile function $\varphi$ and parameter $\beta$ define the same model as the profile function $x\mapsto \varphi(rx)$ and parameter $r\beta$. Hence the normalization convention \eqref{profileNormalization} represents no loss of generality. Similary, if the intensity of the arrival process is taken as $\lambda>0$ the process $(G_{t/\lambda})_{t>0}$ is the original process with the same profile function $\varphi$ and parameter $\beta\lambda$.
	\smallskip
	\item[(ii)]
	The form of the connection probability \eqref{constructionProb} is natural for the following reasons: To ensure that the probability of a new vertex connecting to its nearest neighbour does not degenerate, as $t\rightarrow\infty$, it is necessary to scale $d(x,y)$ by $t^{-1/d}$, which is the order of the distance of a point to its nearest neighbour at time $t$. Further, the integrability condition of $\varphi$ ensures that the expected number of edges connecting a new vertex to the already existing ones, remains bounded from zero and infinity, as $t\rightarrow\infty$.\smallskip
	\item[(iii)]
	In the degree-based spatial preferential attachment model of Jacob and M\"orters~\cite{JacobMoerters2015}, the term $(t/s)^\gamma$ that creates the age dependence in our model is replaced by a function of the indegree, the number of younger vertices $\mathbf{y}$ is connected to at time $t$. If this function is asymptotically linear with slope $\gamma$, the network is scale-free with power-law exponent $\tau=1+\frac1\gamma$. In this case, the expected indegree is of order $(t/s)^\gamma$ so that the models remain comparable and this is the natural choice to ensure that our network model will be scale-free.      
\item[(iv)]
For the profile function $\varphi$, one has different choices. We normally assume that $\varphi$ is either regularly varying at infinity with index $-\delta$, for some $\delta>1$, or $\varphi$ decays quicker than any regularly varying function, in which case we set $\delta=\infty$. In the latter case a natural choice is to consider $\varphi(x)=\frac{1}{2a}\mathbbm{1}_{[0,a]}(x)$ for $a\geq 1/2$. In this case, a vertex born at time $s$ is linked to a new vertex at time $t$ with probability $1/(2a)$ if and only if their positions are within distance $$\left(\sfrac1t \beta a\left(t/s\right)^\gamma \right)^{1/d}.$$ In the case $a=1/2$, the profile function $\varphi$ only takes the values zero and one, thus the decision is not random and we connect two vertices whenever they are close enough. The degree-based preferential attachment model in discrete time for this choice of $\varphi$ was introduced in \cite{AielloEtAll2008} and further studied in \cite{CooperFriezePralat2012} and \cite{JanssenPralatWilson2013}. This particular choice for the profile function helps to get a better understanding of the problems and properties of this model, see for example Section~\ref{Section:Clustering}. However, for our ultimate purpose this choice is too restrictive as it does not allow the networks to be robust or ultrasmall. 
\end{enumerate}

In the following sections, we use $g=o(h)$ to indicate that $g/h$ converges to zero, $g\asymp h$ if $g/h$ is bounded from zero and infinity and $g\sim h$ if $g/h$ converges to one. 

\section{Weak local limit: the age-dependent random connection model} \label{Section:WeakLocalLimit}
In this section, we introduce a graphical representation of the network $G_t$. This representation allows a simple rescaling, and the rescaled graphs turn out to converge to a limiting graph, which is denoted as the \textit{age-dependent random connection model}. This also turns out to be the weak local limit of the graph sequence $(G_t)_{t>0}$, which enables us to achieve results for the network $(G_t)_{t\geq 0}$ by studying the age-dependent random connection model.
\medskip

Let $\mathcal{X}$ denote a Poisson point process of unit intensity on $\R^d\times(0,\infty)$. We say a point $\mathbf{x}=(x,s)\in\mathcal{X}$ is  born at time $s$ and placed at position $x$. Observe that, almost surely, two points of $\mathcal{X}$ neither have the same birth time nor the same position. We say that $(x,s)$ is \textit{older} than $(y,u)$ if $s<u$. For $t>0$ write $\mathcal{X}_t$ for $\mathcal{X}\cap(\mathbb{T}_1^d\times(0,t])$, the set of vertices on the torus already born at time $t$. We denote by 
	\[E(\X):=\set{(\x,\y)\in\mathcal{X} \times \mathcal{X} \colon \x \text{ younger than }\y}\]
the set of \textit{potential edges} in $\X$. Given $\mathcal{X}$ we introduce a family $\V$ of independent random variables, uniformly distributed on $(0,1)$, indexed by the set of potential edges. We denote these variables by $\V_{\x,\y}$ or $\V(\x,\y)$. A realization of $\X_t$ and $\V_t$, defined as the restriction of $\mathcal{V}$ to indices in $\X_t \times \X_t$, defines a network $G(\X_t,\V_t)$ with vertex set $\X_t$ placing an edge between $\x=(x,u)$ and $\y=(y,s)$ with $s<u$, if and only if
	\begin{align}
\		\V(\x,\y)\leq\varphi\left(\frac{u\cdot d(x,y)^d}{\beta\,\left( \frac{u}{s}\right)^\gamma}\right). \label{canonicalConstruction}
	\end{align}
Observe that the graph sequence $(G(\X_t,\V_t))_{t>0}$ has the law of our age-based spatial preferential attachment model and is therefore constructed on the probability space carrying the Poisson process $\X$ and the sequence $\V$. 
Moreover,
$G$~extends to a deterministic mapping associating a graph structure to any locally finite set of points in $\mathcal{Y} \subseteq\mathbb{T}_a^d\times(0,\infty)$ and 
sequence $\mathcal{V}$ in $(0,1)$ indexed by $E(\Y)
=\{ (\x, \y) \in \mathcal{Y}\times\mathcal{Y} \colon \x \text{ younger than }\y \}$, where  \smash{$\mathbb{T}_a^d=(-\frac{1}{2}a^{1/d},\frac{1}{2}a^{1/d}]^d$}
is the torus of volume~$a$ equipped with its canonical metric~$d(\cdot, \cdot)$, and $\x$, $\y$ are connected if and only if \eqref{canonicalConstruction} holds.  We permit the case $a=\infty$, with $\mathbb{T}_\infty^d=\R^d$ equipped with the Euclidean metric.
\pagebreak[3]
\medskip

For finite $t>0$, we define the \textit{rescaling mapping}
\begin{center}
	\begin{tabular}{llll}
		$h_t:$ & $\mathbb{T}_1^d\times(0,t]$ & $\longrightarrow$ & $\mathbb{T}_t^d\times (0,1]$, \\
			   & $(x,s)$					 & $\longmapsto$	     & $\left(t^{1/d}x, s/t\right)$,
	\end{tabular}
\end{center}
which expands space by a factor of $t^{1/d}$ and time by a factor of $1/t$. The mapping $h_t$ operates canonically on the set $\X_t$ as well as on $\V_t$ by $h_t(\V_t)(h_t(\x),h_t(\y)):=\V_t(\x,\y)$, and also on graphs with vertex set in $\X_t$ by mapping points $\x$ to $h_t(\x)$ and 
introducing an edge between $h_t(\x)$ and $h_t(\y)$ if and only if  there is one between $\x$ and $\y$.
As
$$\varphi\left(\frac{u/t\cdot d(t^{1/d}x,t^{1/d}y)^d}{\beta\,\left( \frac{u/t}{s/t}\right)^\gamma}\right)
=\varphi\left(\frac{u\cdot d(x,y)^d}{\beta\,\left( \frac{u}{s}\right)^\gamma}\right)$$
the operation $h_t$ preserves the rule \eqref{canonicalConstruction} and therefore
	\[G(h_t(\X_t),h_t(\V_t))=h_t(G(\X_t,\V_t)).\]
In plain words, it is the same to construct the graph and then rescale the picture, or to first rescale the picture and then construct the graph on the rescaled picture,
see Figure~\ref{Figure:Transformation}.
\medskip

\begin{figure}[!h]
	\begin{center}
		\includegraphics[scale=0.25]{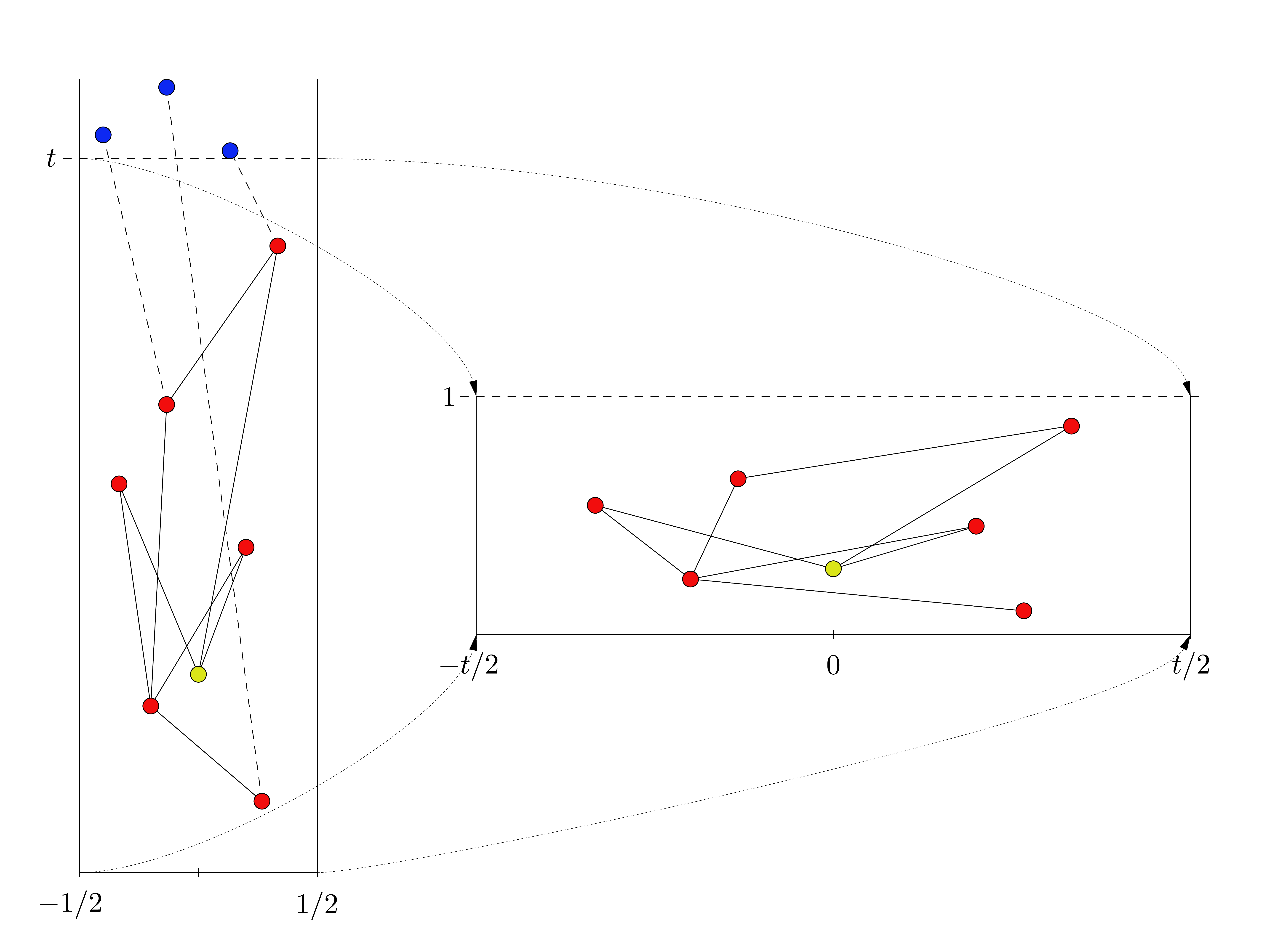}
		\caption{The graph $G_t$ on the left and its rescaling $h_t(G_t)$ on the right. The blue vertices are born after time $t$ and, therefore, the corresponding edges do not exist yet and the vertices are not part of the rescaled graph. The yellow vertex is placed at position $0$ and remains in the centre after the rescaling.}
		\label{Figure:Transformation}
	\end{center}
\end{figure}

 \pagebreak[3]
We now denote $\X^t=\mathcal{X}\cap(\mathbb{T}_t^d\times(0,1])$ and by $\V^t$ the restriction of $\mathcal{V}$ to indices in $\X^t \times \X^t$. This gives rise to a graph  $G^t:=G(\X^t,\V^t)$. As $h_t(\X_t)$ is a Poisson point process of unit intensity on $\mathbb{T}_t^d\times(0,1]$ and $h_t(\V_t)$ are independent uniform marks attached to the potential edges, for fixed finite $t$, the graph $G^t$ has the same law as $G(h_t(\X_t),h_t(\V_t))$ and therefore as $h_t(G_t)$. However, the \emph{process} $(G^t)_{t>0}$ behaves differently from the original process $(G_t)_{t>0}$. Indeed, while the degree of any fixed vertex in $(G_t)_{t>0}$ goes to infinity, 
the degree of any fixed vertex in $(G^t)_{t>0}$ stabilizes and the graph sequence converges to the graph $G^\infty:=G(\X^\infty, \V^\infty)$; see the theorem below.
\smallskip

In order to formulate also a local version of this convergence result we add a point at the origin to our Poisson process denoting $\X_0:=\X\cup \{(0,U)\}$ where $U$ is an independent, uniformly on $(0,1]$ distributed birth time. As before let $\V_0$ be a family of independent uniformly distributed random variables indexed by the potential edges in $\X_0$, and, for $0<t\leq\infty$, let $\X^t_0=\mathcal{X}_0\cap(\mathbb{T}_t^d\times(0,1])$ and denote by $\V_0^t$ the restriction of $\mathcal{V}_0$ to indices in  $\X_0^t \times \X_0^t$. We define rooted graphs $G^t_0:=G(\X_0^t, \V_0^t)$ with the root being the vertex placed at the origin. For $p>0$ define the class  ${\mathcal H}_p$ of nonnegative functions  $h$ acting on locally finite rooted graphs and depending only on a bounded graph neighbourhood of the root with the property that
$$\sup_{0<t<\infty} \E h(G^t_0)^p < \infty.$$

\begin{theorem}  \label{PropADRCM} ~\
	\begin{enumerate}
		\item[(i)] $G^\infty$ is almost surely locally finite, i.e.\ almost surely all its vertices have finite degree.
		\item[(ii)] Almost surely, the graph sequence  $(G^t)$ converges to $G^\infty$ in the sense that for each $\x\in\X^\infty$ the neighbours 
		of $\x$ in $G^t$ and in $G^\infty$ coincide for large $t$.
		\item[(iii)] In probability, the graph sequence  $(G_t)$ converges weakly locally to $G_0^\infty$ in the sense that for any $h\in{\mathcal H}_p$, $p>1$, we have
		\begin{equation}\label{LLN}
		\lim_{t\to\infty} \frac 1t \sum_{\x \in G_t} h( \theta_\x G_t) = \E h ( G_0^\infty) \qquad \mbox{ in probability,} 
		\end{equation}   
		where $\theta_\x$ acts on points $\y=(y,s)$ as $\theta_\x(\y)=(y-x,s)$ and on graphs accordingly.  
	\end{enumerate}	
\end{theorem}

Theorem~\ref{PropADRCM} will be proved in Section~4. The limiting graph $G^\infty$ in $(ii)$ is what we call the \textit{age-dependent random connection model}. This model is of independent interest as a natural generalisation of the random connection model, see Meester and Roy \cite{Meester96} or Last et al \cite{Last2018} for a recent paper. Like in the classical geometric random graph models points are placed according to a Poisson point process $\Y\subseteq\R^d$, but now every point additionally carries a mark drawn independently from the uniform distribution on $(0,1)$. Given points and marks, we independently connect two points in position $x$ with mark $u$, resp.\ position $y$ with mark $s$, with probability 
$$\varphi\big(\beta^{-1}(s\vee u)^{1-\gamma}(s\wedge u)^\gamma \cdot |x-y|^d \big).$$ 
The rooted graph $G_0^\infty$ occurring as the local limit is the \emph{Palm version} of the age-dependent random connection model $G^\infty$; loosely speaking the graph  $G^\infty$ with a typical point shifted to the origin.
\bigskip

{\bf Remarks:} 
\begin{itemize}
	\item Weak local limits were introduced by Benjamini and Schramm~\cite{benjamini2001} as distributional limits for determinstic sequences of finite graphs randomized by
	a uniform choice of root.   The result in~(iii) allows 
that $h$ additionally depends continuously on the ages of the vertices and the length of the edges if taken in the scaled graphs
	$h_t(\theta_\x G_t)$. Further generalisations of the results hold, see Yukich and Penrose~\cite{penrose2003} for seminal work on random geometric graphs
	and Jacob and M\"orters~\cite{JacobMoerters2015} for a similar proof in the case of the degree-based 
model which can be adapted to our situation. We will not need these more general results here.
	\smallskip

	\item The age-dependent random connection model is in a different universality class than other established models of infinite spatial scale-free graphs. For example the scale-free percolation model of Deijfen, van der Hofstad, Hooghiemstra~\cite{DeijfenHofstadHooghiemstra2013} and its continuous version studied by Deprez and W\"uthrich \cite{Deprez2018} correspond (when rewritten in our framework) to a connection probability of the form 
	$$\varphi\big(\beta^{-1} s^\gamma u^\gamma\cdot |x-y|^d \big).$$ 
	Models of this type do not arise naturally from sequences of growing finite random graphs on a fixed space as our model does. 
	\smallskip
	
	\item There is a similar convergence result for the degree-based spatial preferential attachment model, but the limiting graph is not as natural as the age-dependent random connection model as the existence of edges between vertices with given position and age depends in this graph on the existence of edges between the older vertex and other vertices that may lie arbitrarily far away, see Jacob and M\"orters~\cite{JacobMoerters2015}.
	
\end{itemize}

\section{Convergence of neighbourhoods  and degree distributions} \label{Section:DegreeDistr}

In this section we will study the asymptotic degree distribution and show that the age-based spatial preferential attachment model is scale-free. To this 
end we study the neighbourhood of a fixed vertex $\x=(x,u)$ in the graphs $G^t$. We think of edges as oriented from the younger to the older endvertex, so that the \emph{indegree} of $\x$ is the number of younger vertices that connect to it, and the \emph{outdegree} is the number of older vertices it connects to. As our construction is based on Poisson processes and conditionally independent edges, the indegree and outdegree of a fixed vertex are  independent and Poisson 
distributed. 

If $G$ is a graph with vertices in $\T_t^d\times(0,\infty)$, we write $\x\leftrightarrow\y$ to indicate that there is an edge between $\x$ and $\y$ in $G$.
Now, let $\x=(x,u)$ be a vertex in $G$ and define its older neighbours,
\begin{align*}
\Y_\x(G):= \set{\y=(y,s)\in G\colon \x\overset{}{\leftrightarrow}\y, s\leq u}, 
\intertext{and its younger neighbours born before time $s$,} 
\setZ_\x(s,G):=\set{\y=(y,r)\in G\colon \y\overset{}{\leftrightarrow}\x, u<r\leq s}.
\end{align*}
For $t\in(0,\infty]$ and $0<u<s\leq 1$, we write $\Y_\x^t:=\Y_\x(G^t)$ and $\setZ_\x^t(s):=\setZ_\x(s,G^t)$, adding the point $\x=(x,u)$ to the  underlying Poisson process $\mathcal X$ if it is not already there. 
%
\begin{prop}\label{Lemma:InOutprocess} ~\ 
	\begin{enumerate}
		\item[(a)] For every $t\in(0,\infty]$, the older neighbours $\Y_\x^t$ of $\x=(x,u)$ form a Poisson point process on $\mathbb{T}_t^d\times[0,u)$ with intensity measure
		\[\lambda_{\Y_\x^t}:=\varphi\left(\beta^{-1} u\left(\frac{s}{u}\right)^\gamma d(x,y)^d\right)\, dy\, ds.\]
		\item[(b)] For every $t\in(0,\infty]$, the younger neighbours $\setZ_\x^t(s_0)$ of $\x=(x,u)$  at time $s_0\in(u,1]$ form a Poisson point process on $\mathbb{T}_t^d\times(u,s_0]$ with intensity measure
		\[\lambda_{\setZ_\x^t(s_0)}:=\varphi\left(\beta^{-1} s\left(\frac{u}{s}\right)^\gamma d(x,y)^d\right)\, dy\, ds.\]
		\item[(c)] The outdegree of the origin in $G^\infty_0$ is Poisson distributed with parameter~$\frac\beta{1-\gamma}$ and independent of the age $U$ 
		of the origin. 
		\item[(d)] The indegree of the origin in $G^\infty_0$ is mixed Poisson distributed, where the mixing distributuion has the density 
		\begin{align}
		f(\lambda)=\beta^{1/\gamma}(\gamma \lambda +\beta)^{-(1+1/\gamma)}\qquad \mbox{ for $\lambda>0$.} \label{DensityLambda}
		\end{align}
	\end{enumerate} 
\end{prop}
\begin{proof}
	The older neighbours of $\x=(x,u)$ are all neighbours with birth time smaller than $u$, therefore, $\X\cap( \mathbb{T}_t^d\times[0,u))$ is the set of all potential vertices connected to $\x$ by an outgoing edge. Now, given $\X$ a vertex $\y=(y,s)\in\X\cap( \mathbb{T}_t^d\times[0,u))$ is connected to $\x$ independently with probability $\varphi(\beta^{-1} u^{1-\gamma}s^\gamma d(x,y)^d)$. Thus, $\Y_\x^t$ defines a thinning of $\X\cap( \mathbb{T}_t^d\times[0,u))$ and (a) follows. The analogous argument for the vertices in $\X\cap( \mathbb{T}_t^d\times(u,s_0])$ proves (b).
	\smallskip
	
	Applying (a) to $\x=(0,u)$ and $t=\infty$ gives that the number of older neighbours is Poisson distributed with parameter
		\begin{align*}
		\lambda_{\Y_\x^\infty}\big(\R^d\times[0,u]\big) &=\int_0^u\int_{\mathbb{R}^d}  
		\varphi\big(\beta^{-1} u^{1-\gamma}s^\gamma |y|^d\big) \,  dy \, ds 
		= \int_0^u \beta u^{\gamma-1}s^{-\gamma}  \,  ds \int_{\R^d}\varphi\big(|y|^d\big)\,  dy 
		=\frac{\beta}{1-\gamma},								  
		\end{align*} 
		using the normalisation of $\varphi$. The claimed independence follows as the distribution does not depend on~$u$, completing the proof of~(c).
		\smallskip
		
		Applying (b) to $\x=(0,u)$ and $t=\infty$ gives that the number of younger neighbours up to time~$s$ is Poisson distributed with parameter
		\begin{align*}
		\lambda_{\setZ_{\x}^\infty(s)}\big(\R^d\times (u,s]\big) &= \int_u^s \int_{\mathbb{R}^d} \varphi\big(\beta^{-1} v^{1-\gamma}u^\gamma |y|^d\big) \, 
		dy \, dv
		 = \beta\int_u^s  v^{\gamma-1}u^{-\gamma}\int_{\mathbb{R}^d} \varphi (|y|^d)\, dy \, dv \\
		& = \beta\int_u^s v^{\gamma-1}u^{-\gamma} \, dv  = \beta\frac{s^\gamma u^{-\gamma} - 1}{\gamma}.
		\end{align*}
		As $U$ is independent of $\X$ and $\V$ the probability that the indegree equals $k$ is therefore
		\begin{align*}
		\int_0^1 \exp\Big(-\beta\frac{u^{-\gamma}-1}{\gamma}\Big)\cdot\frac{\big(\beta\frac{u^{-\gamma}-1}{\gamma}\big)^k}{k!} \, du 
		& = \int_0^\infty \exp(-\lambda)\cdot\frac{\lambda^k}{k!}\cdot\left(\beta^{1/\gamma}(\gamma \lambda+\beta)^{-(1+1/\gamma)}\right)\, d\lambda,
		\end{align*}
		as claimed.
\end{proof}



\begin{figure}[!h]
	\begin{center}
		\includegraphics[scale=0.3]{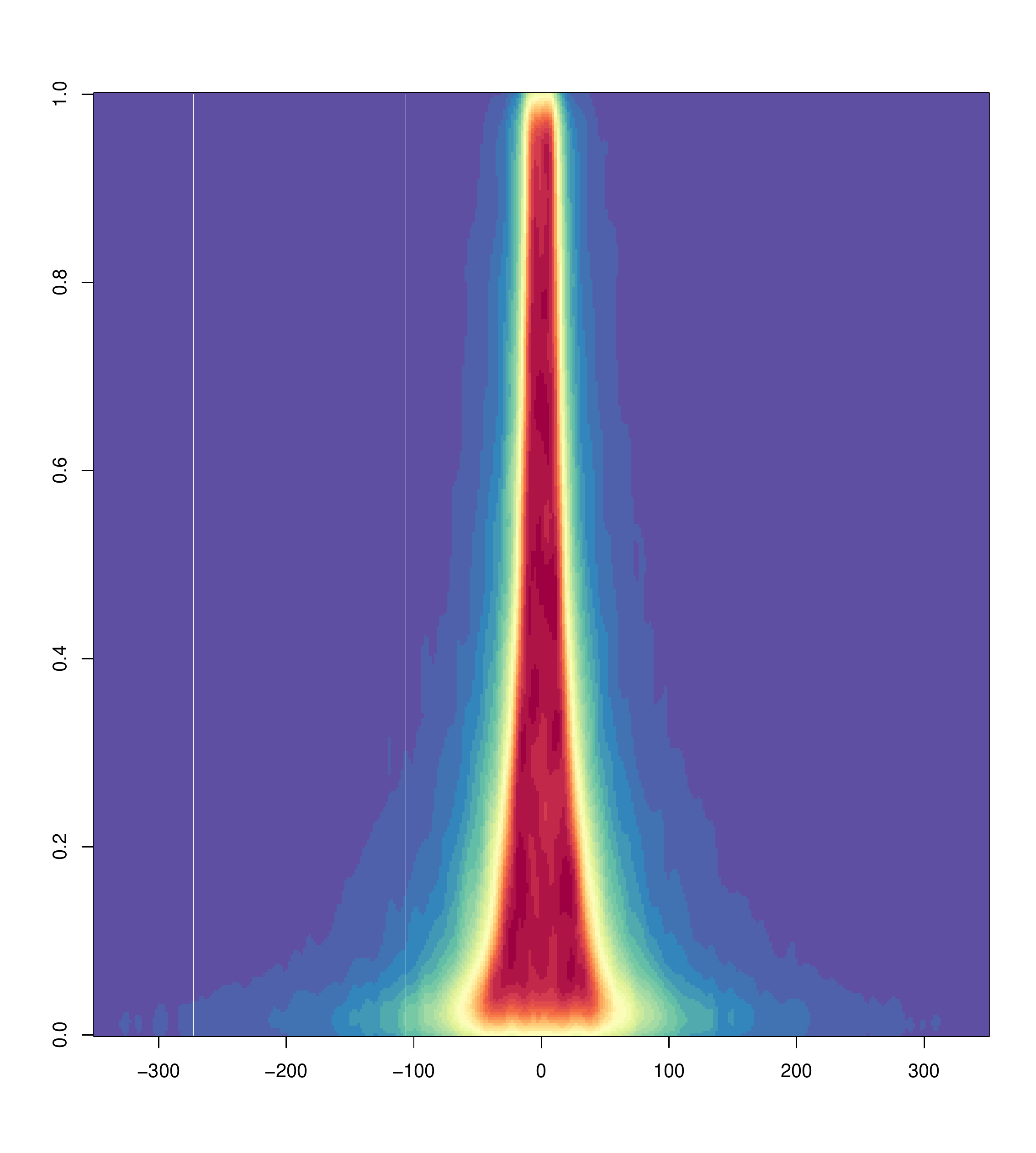}
		\includegraphics[scale=0.3]{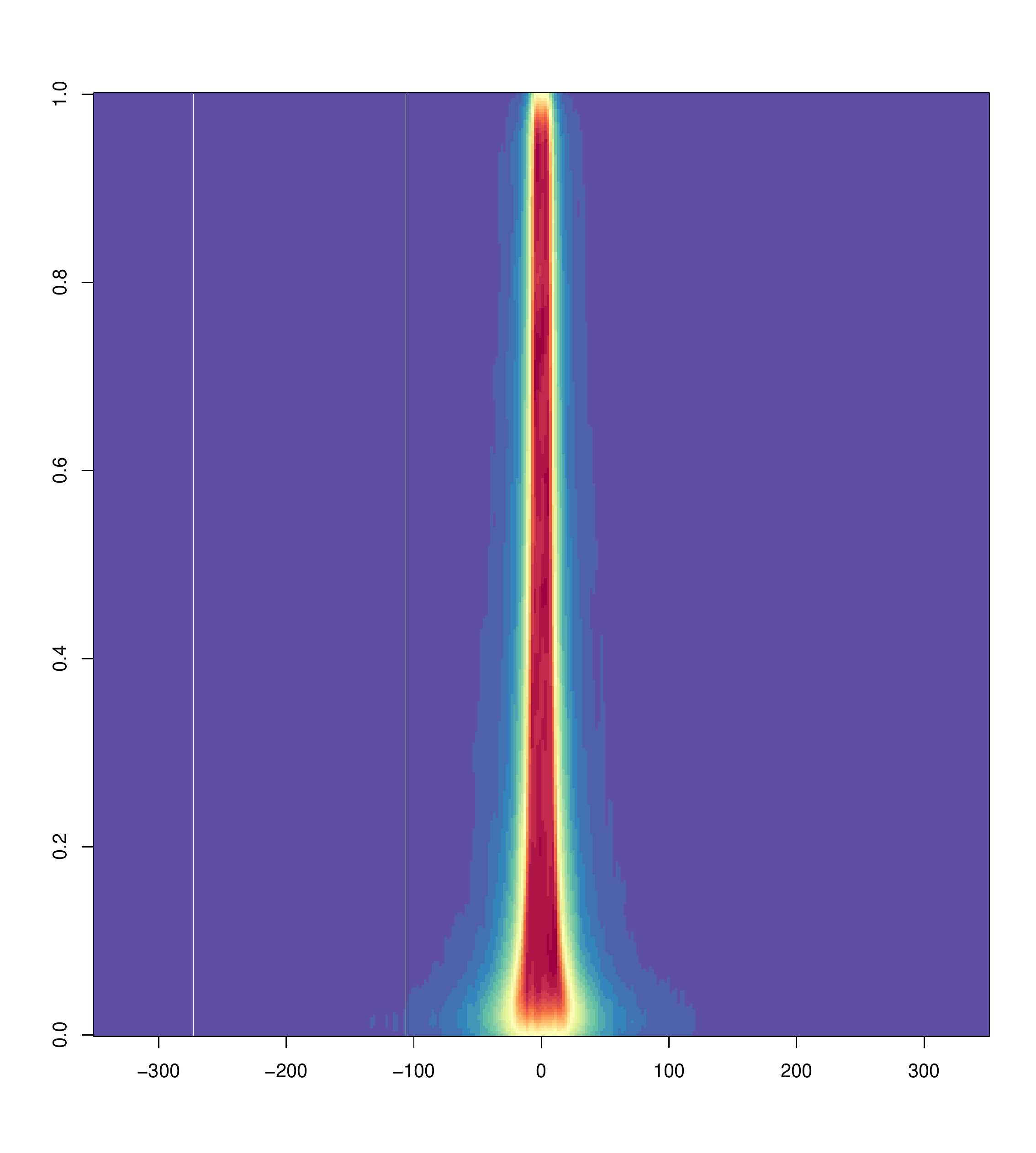}
		\caption{Heatmaps of the neighbourhood of a realtively old root (left, birth time 0.2) and of a relatively young root (right, birth time 0.8) in $G^\infty_0$ with $\beta=5$, $\gamma=1/3$ and $\varphi(x)=1\wedge x^{-2}$. 
		}
		\label{Figure:Neighbourhood}
	\end{center}
\end{figure}

{\bf Remark:}
Since, by construction, $\Y_\x^t$ and $\setZ_\x^t(1)$ are independent Poisson point processes, the neighbourhood of a point $\x=(0,u)$ added (if necessary) to $G^t$ is a Poisson point process with intensity $\lambda_{\setZ_\x^t(1)}+\lambda_{\Y_\x^t}$. Let now $t$ be finite and pick a vertex $X$ uniformly at random from 
the finite graph~$G_t$. We easily see that $h_t(\theta_\x G_t)= G^t_0$ in distribution. Hence Proposition~\ref{Lemma:InOutprocess} part (a) and (b) give a precise description of the neighbourhood of a randomly chosen vertex in $G_t$. 

\begin{proof}[Proof of Theorem~\ref{PropADRCM}(i)]
By Proposition~\ref{Lemma:InOutprocess} part (c) and (d), almost surely, the origin has finite degree in $G^\infty_0$. Hence, by the refined Campbell theorem (see Theorem 9.1 in \cite{LastPenrose2017}), almost surely, every vertex in $G^\infty$ has finite degree.
\end{proof}

\begin{proof}[Proof of Theorem~\ref{PropADRCM}(ii)]
We work conditionally on $\x=(x,s)\in\X^\infty$. Our aim is to show that there exists an almost surely finite random variable $M$ such that, for all $t\in(0,\infty]$ and $\y\in\X^\infty$ with distance at least $M$ from $\x$, the vertices $\x$ and $\y$ are not connected in $G^t$. To this end, observe that the distance between $\x$ and any $\y\in\T_t^d$ can be up to $2\sqrt{d}\abs{x}$ smaller than it would be in $\R^d$. Consider the model where the vertices within distance $2\sqrt{d}\abs{x}$ of $\x$ are deleted from $\mathcal{X}^\infty$ and all the other vertices are moved towards $\x$ by a distance of $2\sqrt{d}\abs{x}$. It is easy to see that all vertices $\y\in\X^\infty$, that are at least $2\sqrt{d}\abs{x}$ away from $\x$ and connected to $\x$ in the finite graph $G^t$ for some $t>0$, are also linked to $\x$ in this new model. Furthermore, the degree of $\x$ is still almost surely finite. Hence, we define the random variable $M$ as the distance of $\x$ to the furthest vertex it is linked to in this new model, plus $2\sqrt{d}\abs{x}$. Then $M$ is almost surely finite and, as for $t>\abs{x}+M$ the vertices in $\X^\infty$ and in $\X^t$ within distance $M$ from $\x$ coincide, the edges of $\x$ linking it to another vertex $\y$ that is at most $M$ away coincide in $G^t$ and $G^\infty$ for sufficiently large $t$. 
\end{proof}

\begin{proof}[Proof of Theorem~\ref{PropADRCM}(iii)]
	We can replace the left-hand side in \eqref{LLN} by the limit of
	$\frac 1t \sum_{\x \in G^t} h( \theta_\x G^t)$, which has the same distribution and, due to Campbell's formula, has expectation
	$\E h ( G_0^t)$. Furthermore, the neighbourhoods of the origin in $G_0^t$ and in $G_0^\infty$ agree for sufficiently large $t$. 
	As the family $(h(G_0^t))_{t>0}$ is bounded in $L^p$ and therefore uniformly integrable, we infer that $\E h ( G_0^t)$ converges to 
	$\E h ( G_0^\infty)$. Hence the first moments in \eqref{LLN} converge, and we now argue that for bounded~$h$  the second moments converge, too.\smallskip

Spelling out the second moment of $\frac 1t \sum_{\x \in G^t} h( \theta_\x G^t)$ we get a term
corresponding to choosing the same $\x\in G^t$ twice, which by the first moment calculation applied to $h^2$ converges to zero, and the term
$$\E\Big[ \frac 1{t^2} \sum_{\heap{\x,\x' \in G^t}{\x\not=\x'}} h( \theta_\x G^t)h( \theta_{\x'} G^t)\Big].$$
Using the boundedness of $h$ we can chose $\eps>0$ so that the contribution from pairs $\x, \x'$ for which one is born before time $\eps$ is arbitrarily small. We can then find a large radius $R$ so that the graph neighbourhood of the origin on which $h$ depends is contained in $\{\y \colon d(0,y)\leq R\}$ for $\theta_\x G^t$ for a proportion of vertices $\x\in G^t$ born after time $\eps$ arbitrarily close to one,
for all sufficiently large~$t$. 
We can neglect the small proportion of exceptional vertices as well as pairs $\x, \x'$ with distance smaller than $R$ using again the boundedness of $h$. On the remaining part the expectation factorizes and we see that second moment converges to $(\E h ( G_0^\infty))^2$. Hence we get convergence in $L^2$.
\smallskip

It remains to remove the condition of boundedness of $h$. Let $k\in\N$ and observe that our result applies to the bounded functional $h\wedge k$. Note that
$$\E \Big[ \frac 1t \sum_{\x \in G^t} h( \theta_\x G^t)-h\wedge k ( \theta_\x G^t) \Big]
= \E \big[ h(G^t_0)-h\wedge k(G^t_0)\big]$$ 
and the right hand side goes to zero uniformly in $t$ as $k\to\infty$ by the uniform integrability implied in our $L^p$ bound. This implies the required convergence. 	
\end{proof}

We define the \textit{empirical outdegree distribution} $\nu_t$ of the graph $G_t$ by
	\begin{equation*}
	\nu_t(k) =
		\frac{1}{t}\sum_{\mathbf{x}\in G_t}\mathbbm{1}_{\{|\Y_\mathbf{x}(G_t)|=k\}}
\qquad  \mbox{ for } k\in\N,
\end{equation*}
and note that (for convenience) we have normalised $\nu_t$ so that its mass converges to one without 
necessarily being equal to one for small $t$.
We now show that the empirical outdegree distribution $\nu_t$ converges to a deterministic limit.
\pagebreak[3]

\begin{theorem} \label{Theorem:OutdegreeDistribution}
	For any function $g:\mathbb{N}_0\rightarrow[0,\infty)$ growing no faster than exponentially 
we have
	\begin{equation*}
		\frac{1}{t}\sum_{\mathbf{x}\in G_t} g\big(|\Y_\mathbf{x}(G_t)|\big)  = \int g \, d\nu_t 
                \longrightarrow\int g \, d\nu,
	\end{equation*}	 
	in probability, as $t\rightarrow\infty$, where $\nu$ is the Poisson distribution with parameter $\beta/(1-\gamma)$. 
\end{theorem}	
\begin{proof}
For a finite graph $G$ with vertices marked by birth times and a root vertex $\x$ 
we can define $h(G)=g(|\Y_{\x}(G)|)$ where $\Y_{\x}(G)$ is the set of
edges from the root to older vertices in~$G$. Note that the function $h$ depends only on the neighbourhood  of the root within graph distance one and the relative birth times of these vertices. Moreover, $h(G_0^t)=g(|\Y_{\x}(G_0^t)|)$ where $\x\in G_0^t$ is the vertex placed at the origin, for arbitrary~$t$, and  as $|\Y_{\x}(G_0^t)|$ are Poisson distributed with a bounded parameter, the integrability condition $h\in\mathcal H_p$ is satisfied as long as $g$ is not growing faster than exponentially. As $h(\theta_\x G_t)=g(|\Y_{\x}(G_t)|)$ for all $\x\in\X_t$ and finite~$t$, 
we infer the result  from Theorem~\ref{PropADRCM}(iii).
\end{proof} 

Define the \textit{empirical indegree distribution} $\mu_t$ of the graph $G_t$ by
\begin{equation*}
	\mu_t(k)= 
		\frac{1}{t}\sum\limits_{\mathbf{x}\in G_t}\mathbbm{1}_{\{ |{\mathcal Z}_\mathbf{x}(t, G_t)|=k\}}.
\end{equation*} 
Similar to above, the empirical indegree distribution $\mu_t$ also converges to a deterministic limit.

\begin{theorem} \label{Theorem:IndegreeDistribution}
For any  function $g:\mathbb{N}_0\rightarrow[0,\infty)$ growing no faster than linearly we have 
\begin{equation*}
	\frac{1}{t}\sum\limits_{\mathbf{x}\in G_t} g\big(|{\mathcal Z}_\mathbf{x}(t, G_t)|\big)
= \int g \, d\mu_t \longrightarrow\int g \, d\mu,
\end{equation*}
in probability, as $t\rightarrow\infty$, where $\mu$ is the mixed Poisson distribution with density $f$ 
as in~\eqref{DensityLambda}
\end{theorem}

\begin{proof}
For a finite graph $G$ with vertices marked by birth times and a root vertex $\x$ 
we can define $h(G)=g(|{\mathcal Z}_{\x}(G)|)$ where ${\mathcal Z}_{\x}(G)$ is the set of
edges from younger vertices in~$G$ to the root. Note that the function $h$ depends only on the neighbourhood  of the root within graph distance one and the relative birth times of these vertices. Moreover, $h(G_0^t)=g(|{\mathcal Z}_\x(G_0^t)|)$ where $\x\in G_0^t$ is the vertex placed at the origin, for arbitrary~$t$. Now $|{\mathcal Z}_{\x}(G_0^t)|$ is dominated by $|{\mathcal Z}_{\x}(G_0^\infty)|$
whose distribution $\mu$ has tails (calculated in Lemma~\ref{Lemma:PowerLaw} below) that vanish fast enough to ensure that  $h\in\mathcal H_p$ for some $p>1$. As $h(\theta_\x G_t)=g(|{\mathcal Z}_{\x}(G_t)|)$ for all $\x\in\X_t$ and finite~$t$ we infer the result  from Theorem~\ref{PropADRCM}(iii).
\end{proof}

To complete the proof that the age-based preferential attachment model  is scale-free with power law 
exponent $\tau=1+1/\gamma$ we observe that, by a similar argument as in Theorem~\ref{Theorem:OutdegreeDistribution} and Theorem~\ref{Theorem:IndegreeDistribution}, the empirical degree distribution in $G_t$ converges in probability to the convolution of $\nu$ and $\mu$. As $\nu$ has superexponentially light tails, the tail
behaviour of the convolution is inherited from that of $\mu$, which we now calculate. 

\begin{lemma}\label{Lemma:PowerLaw}
$\mu(k)=k^{-(1+\frac{1}{\gamma})+o(1)}$ as $k\uparrow\infty$.
\end{lemma}

\begin{proof}

Observe that
	\begin{align*}
		\mu(k) & = \beta^{1/\gamma}\int_0^\infty \frac{\lambda^k}{k!}e^{-\lambda}(\gamma\lambda+\beta)^{-(1+\frac{1}{\gamma})} \, d\lambda 
\leq \frac{\beta^{1/\gamma}\gamma^{-1-\frac{1}{\gamma}}}{\Gamma(k+1)}\int_0^\infty \lambda^{(k-\frac{1}{\gamma})-1}e^{-\lambda} \, d\lambda\\ &  =\frac{\beta^{1/\gamma}}{\gamma^{1+1/\gamma}}\, \frac{\Gamma(k-\frac{1}{\gamma})}{\Gamma(k+1)} 
= k^{-1-\frac{1}{\gamma}+o(1)},
	\end{align*}
as $k\uparrow\infty$, by Stirling's formula. On the other hand, note that for some fixed bound $A> 0$, there exists a constant $c>0$ such that $\gamma x+\beta\leq c\gamma x$ for all $x\geq A$. Hence
	\begin{align*}
		\mu(k)&\geq \frac{c^{-1-\frac{1}{\gamma}}\beta^{1/\gamma}}{\Gamma(k+1)}\int_A^\infty \lambda^k e^{-\lambda}(\gamma\lambda)^{-1-\frac{1}{\gamma}}\, d\lambda  
= \tilde{c}\,\frac{\Gamma(k-\frac{1}{\gamma})}{\Gamma(k+1)} - \frac{\tilde{c}}{\Gamma(k+1)} \int_0^A \lambda^{(k-\frac{1}{\gamma})-1} e^{-\lambda}\, d\lambda, 
	\end{align*}
for some positive constant $\tilde{c}$. As the subtracted term is of smaller order we obtain the
lower bound.
%
\end{proof}

\section{Global and local clustering coefficients} \label{Section:Clustering}

To show that the age-based spatial preferential attachment model has clustering features we use two metrics well established in the applied networks literature, see e.g.~\cite{watts98, Newman01} for some early papers. If $G$ is a finite graph, we call a pair of edges
in $G$ a \emph{wedge} if they share an endpoint (called its \emph{tip}). Define the \textit{global clustering coefficient} or \emph{transitivity} of $G$ as
\begin{align*}
	c^{\text{glob}}(G) := 3 \frac{\text{Number of triangles in } G}{\text{Number of wedges in } G},
\end{align*}
if there is at least one wedge in $G$ and $c^{\text{glob}}(G):=0$ otherwise. By definition, $c^{\text{glob}}(G)\in[0,1]$. \smallskip

Another way of thinking about clusters is locally; i.e. to count only the triangles and wedges containing a fixed vertex $\mathbf{x}$. For a vertex $\x$ with at least two neighbours, define the \textit{local clustering coefficient} by
\begin{align*}
	c_\x^{\text{loc}}(G):=\frac{\text{Number of triangles in } G \text{ containing vertex } \x}{\text{Number of wedges with tip } \x \text{ in } G},
\end{align*}
which is also an element of $[0,1]$. Let $V_2(G)\subseteq G$ be the set of vertices in $G$ with degree at least two, and define the \textit{average clustering coefficient} by
\begin{align*}
	c^\text{av}(G):=\frac{1}{\abs{V_2(G)}}\sum_{\x \in V_2(G)}c_\x^{\text{loc}}(G),
\end{align*}
if $V_2(G)$ is not empty and as $c^{\text{av}}(G):=0$ otherwise. Note that this metric places more weight on the low degree nodes, while the transitivity places more weight on the high degree nodes.
\smallskip

\begin{theorem}[Clustering Coefficients]\label{Theorem:Clustering} ~\
	\begin{enumerate}
		\item[(a)] For the average clustering coefficient we have
			\begin{align*}
				c^\text{av}(G_t)\longrightarrow \int_0^1 \P\big\{ (X_u^{\ssup 1},S_u^{\ssup 1})\overset{}{\leftrightarrow}(X_u^{\ssup 2},S_u^{\ssup 2})\big\} \, \pi(du),
			\end{align*}
			in probability as $t\rightarrow\infty$, where $(X_u^{\ssup 1},S_u^{\ssup 1})$ resp.~$(X_u^{\ssup 2}, S_u^{\ssup 2})$ are two independent random variables on $\R^d\times[0,1]$ with distribution
				\begin{align}
			\frac{1}{\lambda_u}  \Big(\varphi(\beta^{-1} s^{1-\gamma}u^\gamma{|x|}^d)\mathbbm{1}_{(u,1]}(s)+\varphi(\beta^{-1} u^{1-\gamma}s^{\gamma}{|x|}^d)\mathbbm{1}_{[0,u]}(s)\Big) \, dx \, ds, \label{clusteringLimitDesnity}
			\end{align}
			where \smash{$\lambda_u=\frac\beta\gamma(\frac{2\gamma-1}{1-\gamma}+u^{-\gamma})$} is the normalising factor, 
			and $\pi$~is the probability measure on $[0,1]$ with density proportional to $1-e^{-\lambda_u}-\lambda_u e^{-\lambda_u}$.\smallskip
		\item[(b)] For the global clustering coefficient, there exists a number 
$c_\infty^\text{glob}\geq 0$ such that
		\begin{align*}
			c^\text{glob}(G_t)\longrightarrow c_\infty^\text{glob}
		\end{align*}
		in probability, as $t\rightarrow\infty$. The limiting global clustering coefficient $c_\infty^\text{glob}$ is positive if and only if $\gamma<1/2$.
	\end{enumerate}
\end{theorem}

\pagebreak[3]

{\bf Remark:} The limiting  average clustering coefficient can be interpreted as the probability that in $G^\infty_0$
two neighbours of the vertex at the origin are connected by an edge. The density of the birthtime of the vertex at the origin 
here is not uniform but given by the measure $\pi$, which is the conditional distribution of the birthtime of a vertex given that it has 
degree at least two. Observe that this coefficient is always positive.  By contrast the global clustering coefficient vanishes asymptotically  when preferential attachment to old nodes is strong (i.e.\ when $\gamma$ is large). In this case the collection of wedges is dominated by those with an untypically old tip. These vertices have small local clustering as they are endvertices to a  significant amount of long edges.
\medskip

%

\begin{proof} 
Let $G$ be a finite rooted graph and define the function $h(G)=c_\x^{\text{loc}}(G)$ if the root $\x$ has degree at least two, and $h(G)=0$ otherwise. 
As $h$ is bounded, we have $h\in\mathcal{H}_p$ for any $p>1$ and, by Theorem~\ref{PropADRCM}~(iii), we get
	\[\frac{1}{t} \sum_{\x\in G_t} h(\theta_\x G_t) \longrightarrow \E\left[h(G_0^\infty)\right]\]
in probability, as $t\rightarrow\infty$. To calculate the limit, observe that, for a vertex $\x$ with degree $k$, the number of wedges with tip $\x$ is $k(k-1)/2$.  It follows that
\begin{align*}
\E\left[h(G_0^\infty)\right] 
& = \int_0^1\sum_{k\geq 2} \E\bigg[\frac{2}{k(k-1)}\sum_{(x,s)\overset{}{\leftrightarrow}(0,u)}\sum_{\underset{v<s}{(y,v)\overset{}{\leftrightarrow}(0,u)}}\mathbbm{1}_{\set{(x,s)\overset{}{\leftrightarrow}(y,v)}} \1_{\{\vert\Y_{(0,u)}^\infty\vert+\vert\setZ_{(0,u)}^\infty(1)\vert = k \}} \bigg] \,  du.
\end{align*}
By Proposition~\ref{Lemma:InOutprocess}, the neighbourhood of the root $(0,u)$ is given by a Poisson point process with intensity measure 
$$\lambda_{{\mathcal Z}_{(0,u)}^\infty(1)}+\lambda_{ \Y_{(0,u)}^\infty}.$$ Conditioned on the number of neighbours, the 
neighbours of the root $(0,u)$ are independent and identically distributed by the normalized intensity measure of the neighbourhood given in \eqref{clusteringLimitDesnity}, see \cite[Proposition 3.8]{LastPenrose2017}. Therefore,
\begin{align*}
\E\left[h(G_0^\infty)\right] &  = \int_0^1 \P\big\{ (X_u^{\ssup 1},S_u^{\ssup 1})\overset{}{\leftrightarrow}(X_u^{\ssup 2},S_u^{\ssup 2})\big\} \,  
\P \big\{\vert\Y_{(0,u)}^\infty\vert+\vert\setZ_{(0,u)}^\infty(1)\vert \geq 2 \big\} \, 
du,\end{align*}
where $(X_u^{\ssup 1},S_u^{\ssup 1})$ and $(X_u^{\ssup 2}, S_u^{\ssup 2})$ are independent and identically distributed as claimed. Choosing
$h(G)$ as the indicator of the event that the root has degree at least two, Theorem~\ref{PropADRCM}~(iii) gives
$$\frac{\vert V_2(G_t)\vert}{t} \longrightarrow \int_0^1 \P \big\{\vert\Y_{(0,u)}^\infty\vert+\vert\setZ_{(0,u)}^\infty(1)\vert\geq 2 \big\} \, du,$$
in probability. As $\vert\Y_{(0,u)}^\infty\vert+\vert\setZ_{(0,u)}^\infty(1)\vert$ is Poisson distributed with intensity $\lambda_u$ we conclude that
\[c^{\text{av}}(G_t) \longrightarrow 
\frac{\int_0^1 \P\big\{ (X_u^{\ssup 1},S_u^{\ssup 1})\overset{}{\leftrightarrow}(X_u^{\ssup 2},S_u^{\ssup 2})\big\} \,  
\big(1-e^{-\lambda_u}-\lambda_u e^{-\lambda_u}\big) \, du}{\int_0^1 1-e^{-\lambda_u}-\lambda_u e^{-\lambda_u} \, du},\]
as claimed in part~(a).\smallskip

For the global clustering coefficient, we count the number of triangles and wedges separately. To this end, define 
$h(G)$ to be the number of triangles which have their youngest vertex in the root of $G$, and $\hat{h}(G)$ to be the number 
of wedges with tip in the root $\x$ of $G$. Note that $h(G_0^t)\leq \vert\Y_\x(G_0^\infty)\vert^2$ and thus $h\in\mathcal{H}_p$ for any $p>1$. 
Moreover,
	\[\hat{h}(G_0^t)=\frac{1}{2}\vert\Y_\x^t\vert(\vert\Y_\x^t\vert-1)+\frac{1}{2}\vert\setZ_\x^t(1)\vert(\vert\setZ_\x^t(1)\vert-1)+\vert\Y_\x^t\vert\vert\setZ_\x^t(1)\vert \leq 2 \big( \vert\Y_\x^\infty\vert^2 + \vert\setZ_\x^\infty(1)\vert^2\big) .\]
If $\gamma<1/2$ and $1<p<1/(2\gamma)$, we hence have $\hat{h}\in\mathcal{H}_p$ and Theorem~\ref{PropADRCM}(iii) gives that
	\[ c^\text{glob}(G_t)=\frac{\sum_{\x\in G_t}h(\theta_\x G_t)}{t}\cdot\frac{t}{\sum_{\x\in G_t}\hat{h}(\theta_\x G_t)}\longrightarrow\frac{\E[h(G_0^\infty)]}{\E[\hat{h}(G_0^\infty)]}>0\]
in probability. If $\gamma> 1/2$, applying the theorem to the bounded functions $\hat{h}(G_t)\wedge k$ and then sending $k$ to $\infty$, we get \smash{$\frac{1}{t}\sum_{\x\in G_t}\hat{h}(\theta_\x G_t)\rightarrow\infty$} and hence $c^\text{glob}(G_t)\rightarrow 0$ in probability, as $t\rightarrow\infty$.
\end{proof}
\medskip

\begin{figure}[h]
\begin{center}
\includegraphics[width = 0.49\linewidth]{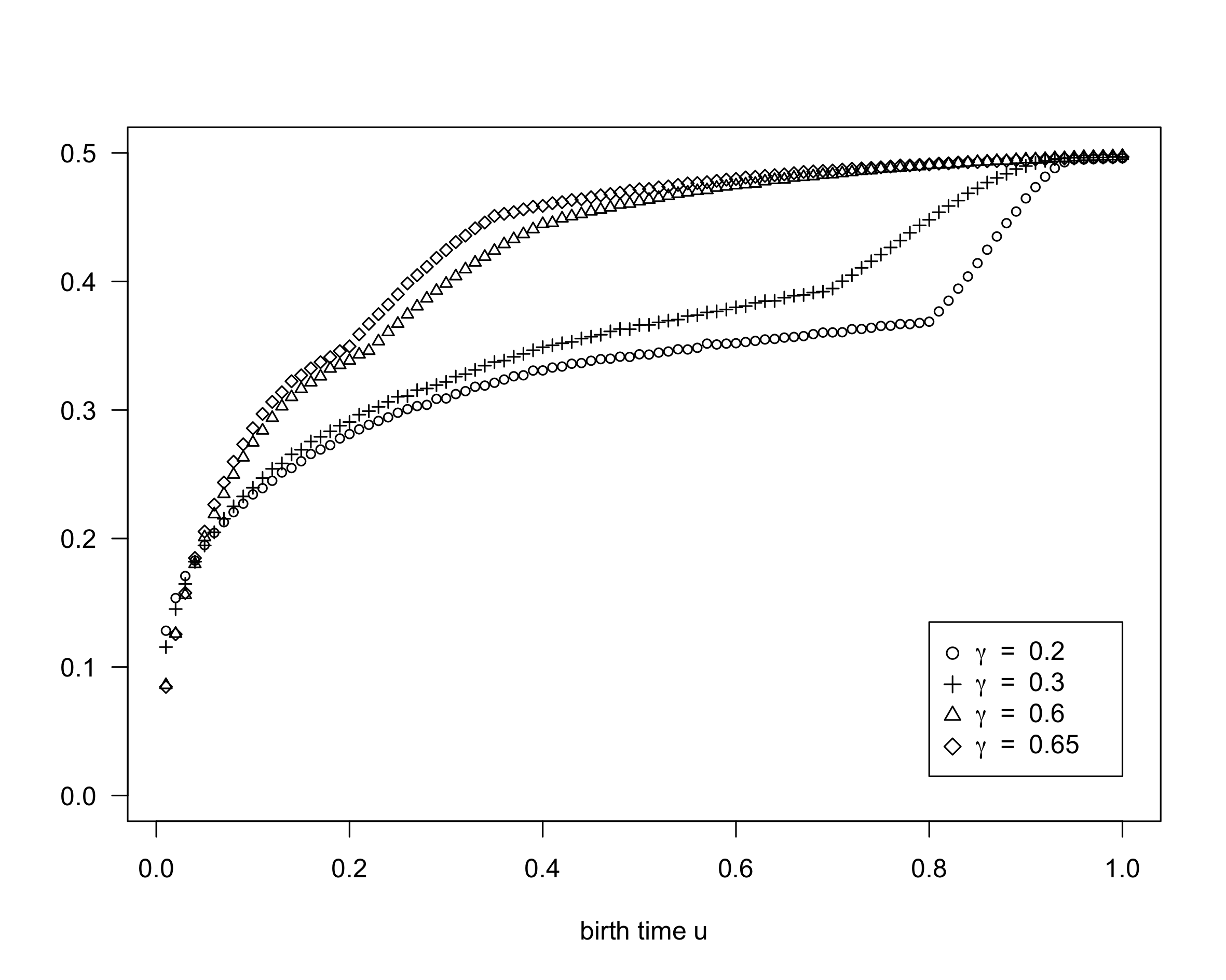}
\includegraphics[width = 0.49\linewidth]{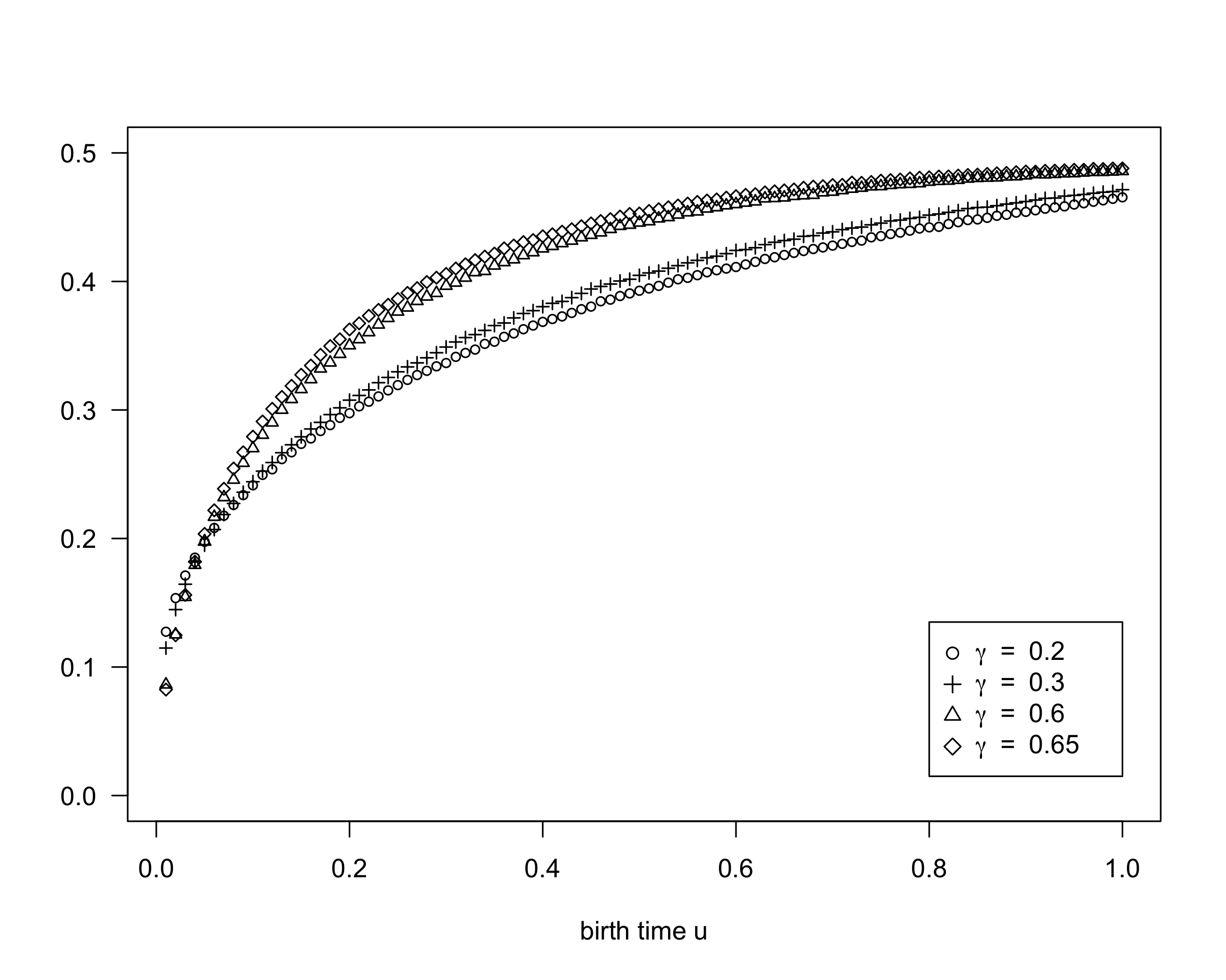}
\caption{Local clustering coefficient of a vertex $(0,u)$ for parameters $a = 1$ and $\beta = c_{\textrm{ed}}(1-\gamma)$ chosen such that the asymptotic edge density is fixed at $c_{\textrm{ed}}$. The left plot corresponds to the case with low edge density ($c_{\textrm{ed}} = 0.1$), while the right plot corresponds to high edge density ($c_{\textrm{ed}} = 10$). The graph on the left clearly shows the crossover in the parameter $u$, depending whether the tip of the wedge is predominantly the youngest, middle or oldest of the three involved vertices.}
\label{Img:Localclustering_varyingu}
\end{center}
\end{figure}
The local and average clustering coefficients cannot be calculated explicitly, but can be simulated; see the appendix of this paper for a discussion on the simulation techniques used here.
 We focus on the profile functions $\varphi=\frac{1}{2a}\mathbbm{1}_{[0,a]}$, for $a\geq 1/2$, dimension $d=1$, and fixed edge density $\beta/(1-\gamma)$.  
Figure~\ref{Img:Localclustering_varyingu} shows the local clustering coefficient of a vertex of age $u$ in $G^\infty$ showing monotone dependence on the age, i.e.\ the empirical probability that two neighbours of a given vertex are connected to each other is larger for younger vertices. This coincides with our intuitive understanding of the local structure of the networks, in which a young vertex, typically, is connected to either very close or very old vertices such that two randomly chosen neighbours have a decent chance of being connected to each other as well. By contrast,  an old vertex typically has more long edges to younger vertices. Thus, two of its neighbours are typically further apart, which reduces the chance of them being each others neighbour. This monotonicity occurs independently of the choice of $\beta$, $\gamma$ and $a$. 
\medskip

\begin{figure}[!h]
\begin{center}
\includegraphics[width = 0.49\linewidth]{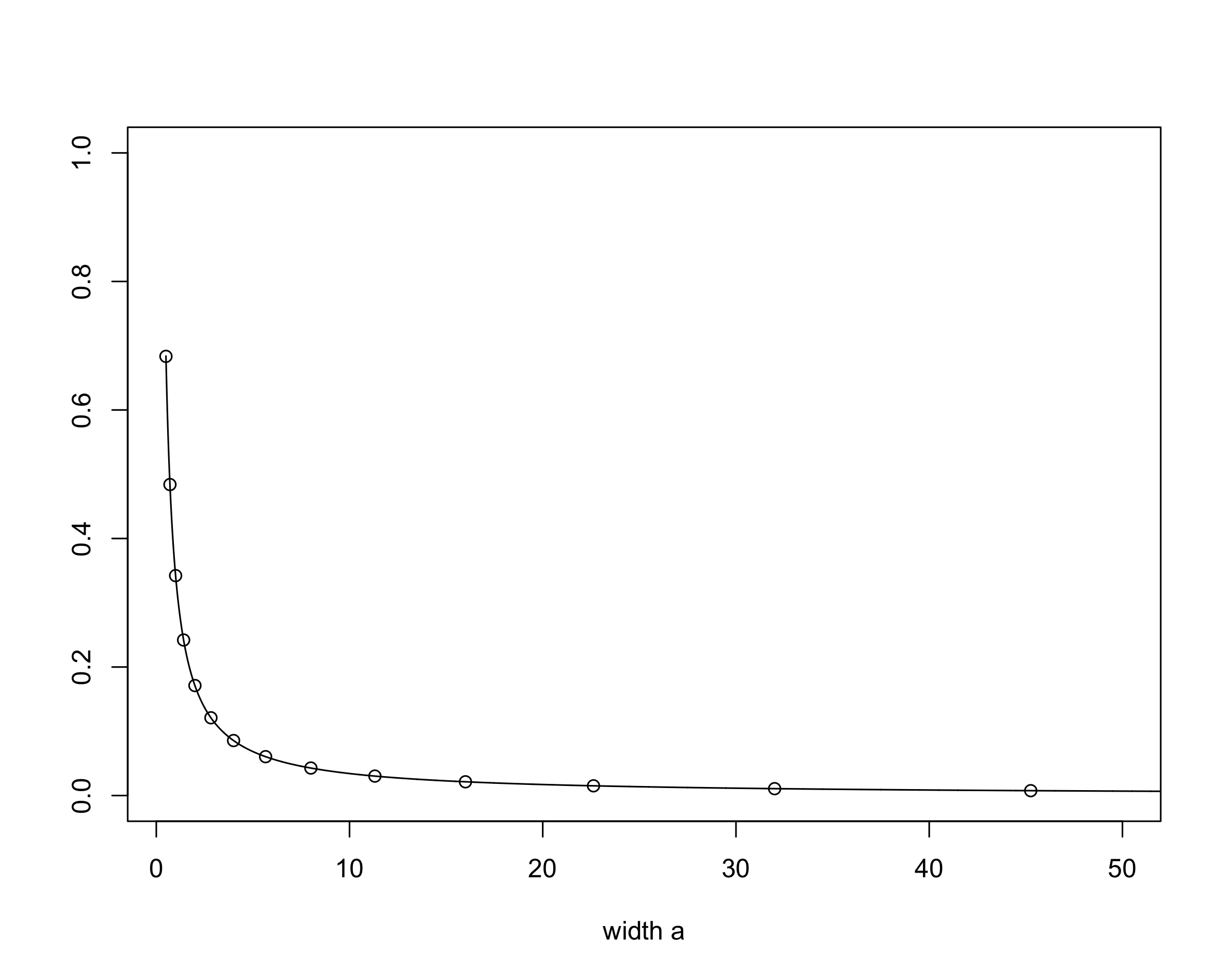}
\includegraphics[width = 0.49\linewidth]{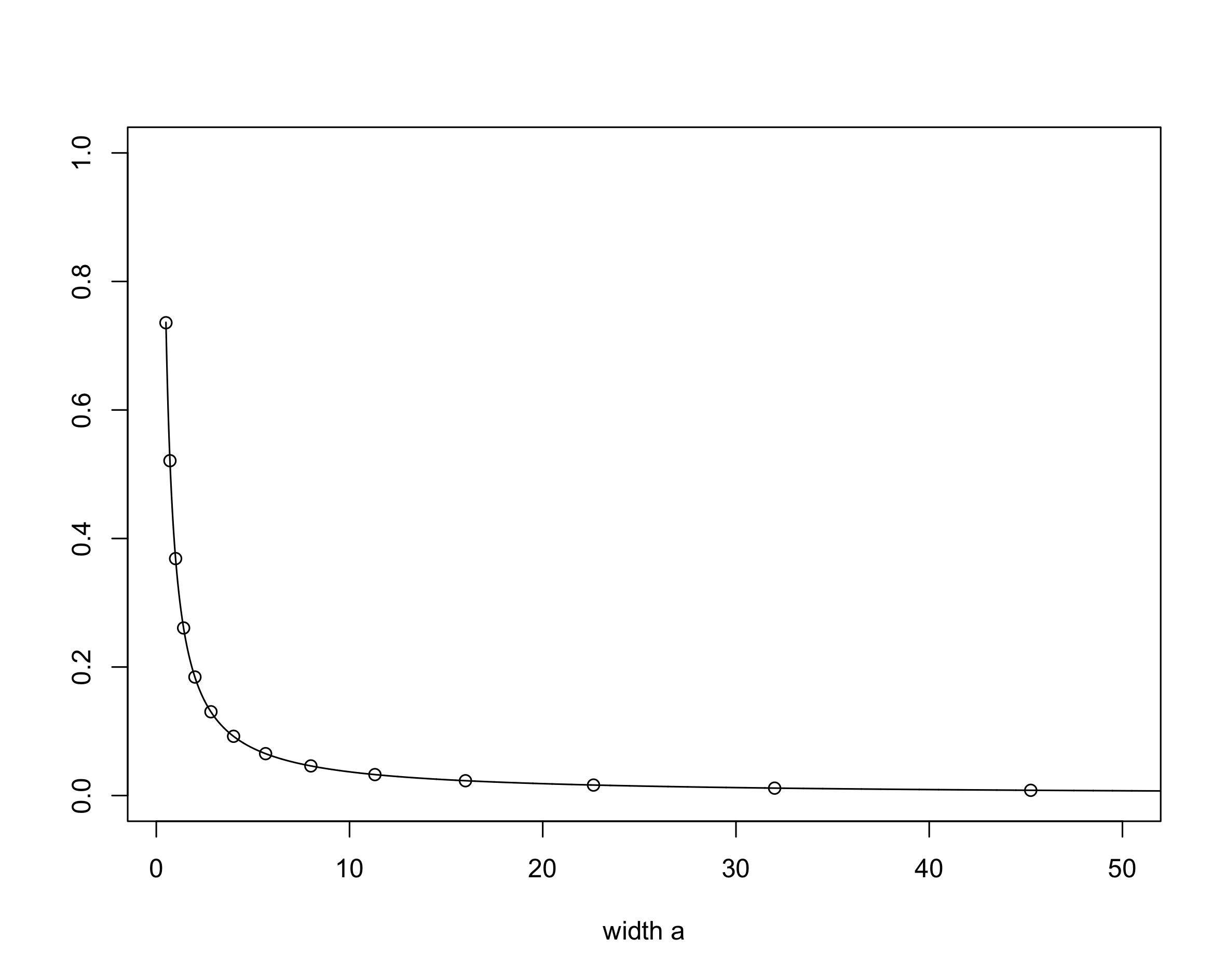}\\
\includegraphics[width = 0.49\linewidth]{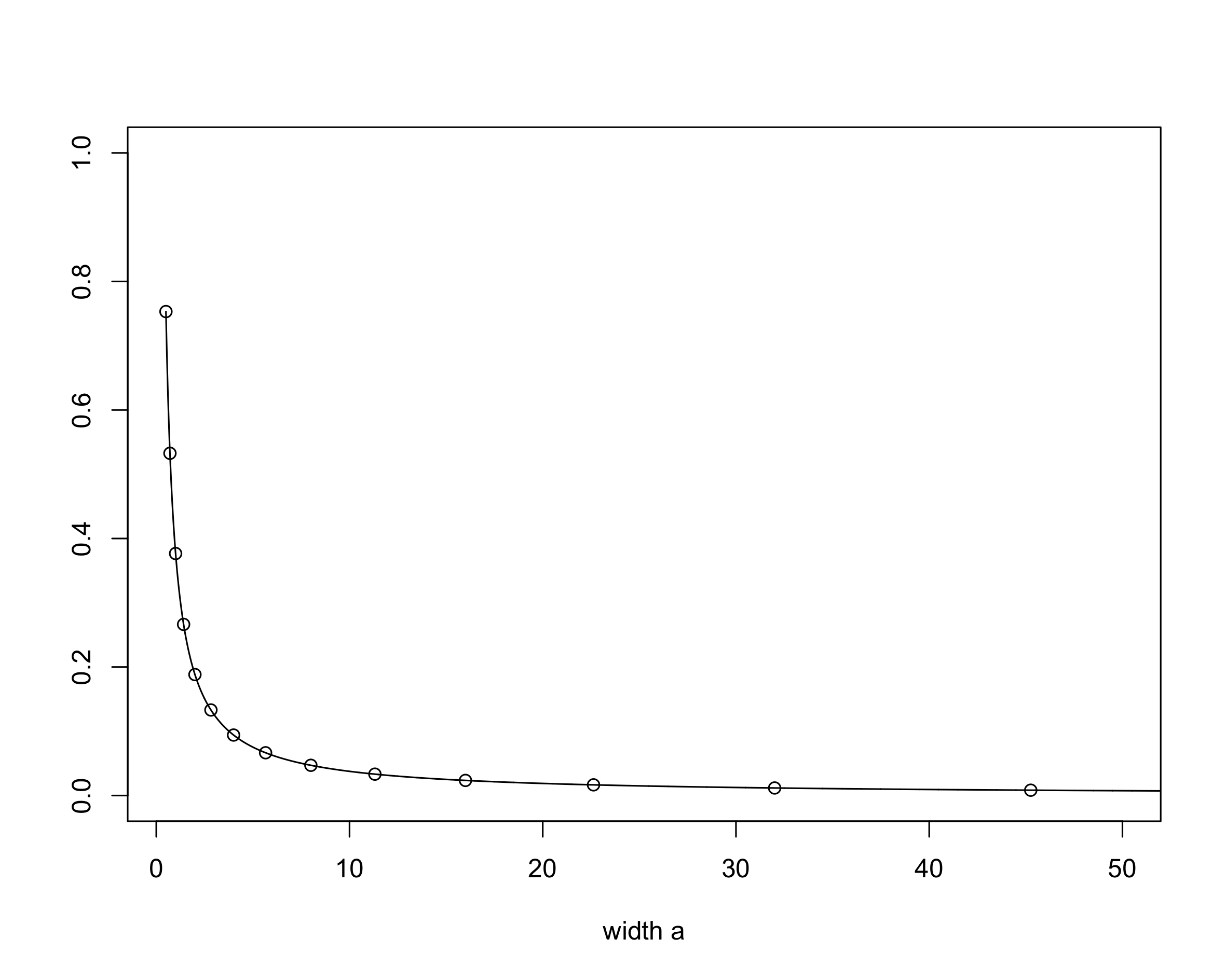}
\includegraphics[width = 0.49\linewidth]{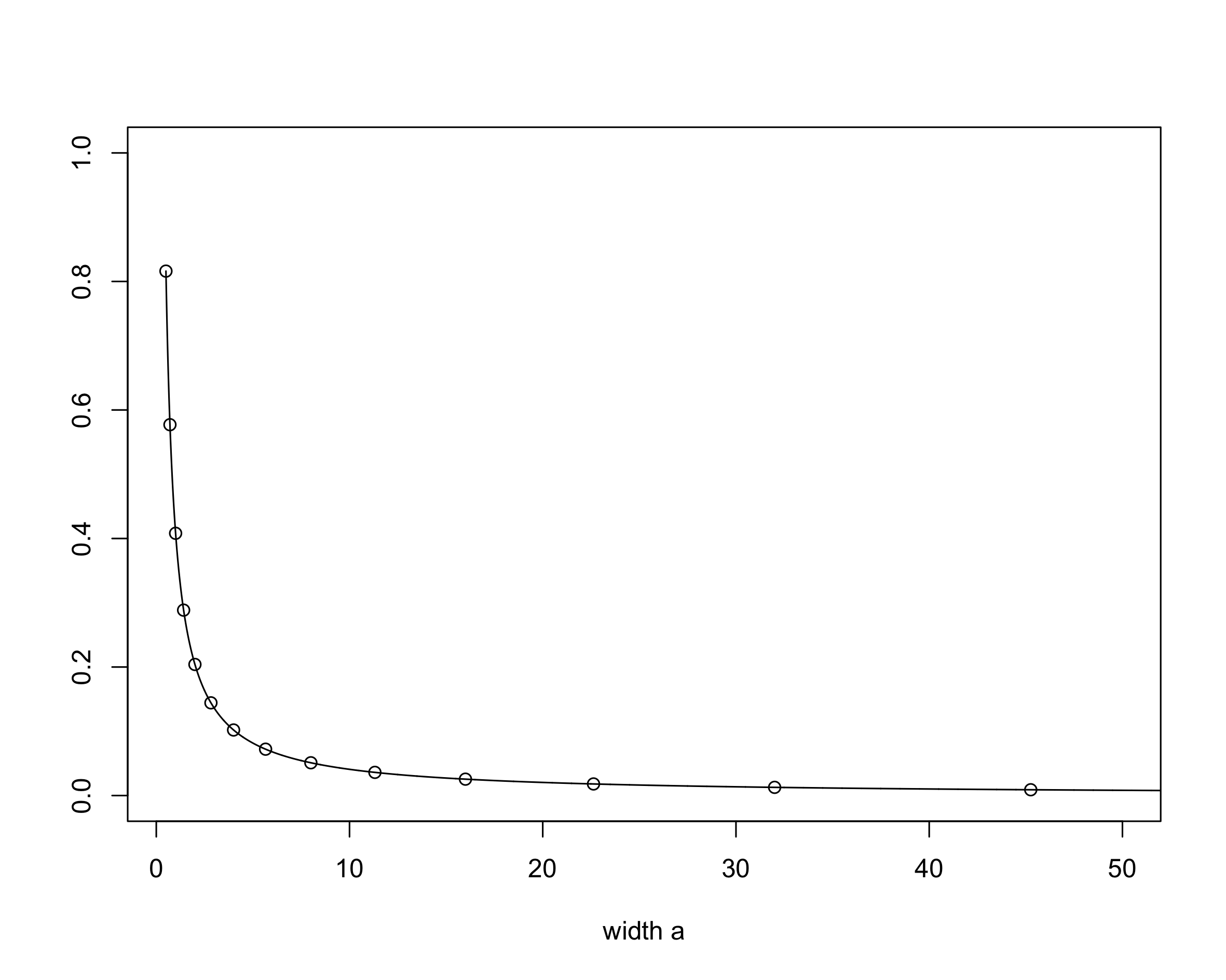}
\caption{Average clustering coefficient for the network with 
profile function $\varphi=\frac{1}{2a}\mathbbm{1}_{[0,a]}$ plotted against the width $a$, 
for $\gamma = 0.3$ in the left resp.\ $\gamma = 0.6$ in the right graphs. The graphs in the top row correspond to fixed edge density 1 while the bottom row corresponds to edge density 10.}
\label{Img:Localclustering_randomU}
\end{center}
\end{figure}

In Figure~\ref{Img:Localclustering_randomU} we see that the dependence of the average clustering coefficient with respect to the width~$a$ of the profile function is  of order $\frac{1}{a}$, a scaling that we also see in the analysis of the global clustering coefficient in the case $\gamma<\frac12$. 
Hence, the average clustering coefficient and the global clustering coefficient (if $\gamma<\frac12$) can be varied by the choice of $\varphi$ and can be made arbitrarily small by choosing $a$ large. Unlike with the global clustering coefficient, there is a mild dependence on $\beta$. 
Again, roughly speaking, large width of $\varphi$ encourages long edges and reduces clustering.
\medskip



\section{Asymptotics for typical edge lengths} \label{Section:EdgeLength}
In this section we study the distribution of the length of typical edges in $G_t$. 
We denote by $E(G)$ the set of edges of the graph $G$ and define $\lambda_t$, the (rescaled) \emph{empirical edge length distribution} in $G_t$, by
$$\lambda_t =
		\frac{1}{|E(G_t)|}\sum_{(\mathbf{x},\mathbf{y})\in E(G_t)}\delta_{t^{1/d} d({x},{y})}. $$

\begin{theorem}\label{Theorem:EdgeLengthDistribution}
	For every continuous and bounded $g\colon[0,\infty) \to \mathbb{R}$, we have
		\[
		\frac{1}{|E(G_t)|}\sum_{(\mathbf{x},\mathbf{y})\in E(G_t)} g\big(t^{1/d} d({x},{y}) \big) = 	\int g \,  d\lambda_t \to \int g \, d\lambda,
		\]
		in probability, as $t\to \infty$, where the limiting probability measure $\lambda$ on $(0,\infty)$ is given by
\begin{align}
\lambda([a,b)) = \frac{1-\gamma}{\beta}\int_0^1 \int_0^u\int_{a\leq|y|<b} \varphi\big(\beta^{-1} u^{1-\gamma}s^{\gamma}|y|^d\big)\, dy \, ds \, du.\label{edgeLengthLimitMeasure}
\end{align}
\end{theorem}

\pagebreak[3]

\begin{proof}
For a finite graph $G$ with vertices positioned in $\R^d$ and marked by birth times and with a root vertex $\x$ placed at the origin define, for $a<b\in[0,\infty]$, 
the function
	\begin{align}h_{a,b}(G)=\sum_{\y\in\Y_\x(G)}\mathbbm{1}_{[a,b)}(\abs{y}).\label{edgeLengthFunctional}\end{align}
Observe that the law of $\lambda_t([a,b))$ in $G_t$ equals the law of 
\[ \frac{1}{|E(G^t)|}\sum_{\x \in \X^t} h_{a,b}(\theta_\x G^t).\]
As mentioned in the remark following the theorem,  Theorem~\ref{PropADRCM} is applicable to functions $h_{a,b}$ depending 
on the length of edges in the rescaled graphs $(G^t)_{t>0}$.  Since the sum in \eqref{edgeLengthFunctional} is dominated by the outdegree, $h_{a,b}\in\mathcal{H}_p$ for some $p>1$. We thus get
$$ \frac{1}{t}\sum_{\x \in \X^t} h_{a,b}(\theta_\x G^t) \longrightarrow \E[h_{a,b}(G_0^\infty)],$$
and since  Theorem~\ref{PropADRCM}~(iii) also gives $\abs{E(G^t)}/t\rightarrow \frac{\beta}{1-\gamma}$ 
and $\lambda([a,b)]=\frac{1-\gamma}{\beta}\E[h_{a,b}(G_0^\infty)]$ we infer that
	\smash{$\lambda_t([a,\infty))\longrightarrow\lambda([a,\infty))$}
in probability, as $t\rightarrow\infty$. Therefore, convergence in probability of $\lambda_t$ to $\lambda$ in the space of probability measures on $\R_+$, equipped with the L\'e{}vy-Prokhorov metric follows.
\end{proof}\smallskip

{\bf Remark:} Suppose there exists $\delta > 1$ such that the profile function satisfies $\varphi(x^d) \asymp 1 \wedge x^{-d\delta}$. Then the explicit formula for $\lambda$ can be used to calculate the tail behaviour of $\lambda$. A straightforward calculation gives that 
\smash{$\lambda([K,\infty)) \asymp 1 \wedge (\beta^{-1/d} K)^{-\eta},$} 
where 
\begin{equation}\label{etadef}\eta:= \min\big\{ d, d(\sfrac1\gamma - 1), d(\delta -1) \big\}.\end{equation}  
In particular, $\lambda$ has finite expectation if $\eta>1$ and infinite expectation if $\eta<1$.  \medskip

We denote by $M_0^\infty$ the length of the longest outgoing edge of the origin in $G^\infty_0$. By the construction of $\lambda$ above, $\lambda([K,\infty)]$ is the expected number of outgoing edges of length bigger than $K$ divided by the total number of outgoing edges from the origin. If $K$ is large this should be of similar order to the probability that $M_0^\infty\geq K$.  This is confirmed in the following lemma.

\begin{lemma}\label{Lemma:ExpectedMaximalEdgeLength}
$\E\left[(M_0^\infty)^a\right]$ is finite if $a<\eta$ and infinite if $a>\eta$, 
where $\eta$ is as defined in \eqref{etadef}.
\end{lemma}

\begin{proof}
We show that the tail probability $\P\set{(M_0^\infty)^a \geq K}$ is of order $K^{-\eta/a}$ 
as $K\to \infty$. The number of outgoing edges with length at least $K^{1/a}$
in $G^\infty_0$ from the vertex $(0,u)$ at the origin  are Poisson distributed with  parameter 
$$\lambda_{K^{1/a},u}:=\lambda_{\Y_{(0,u)}^\infty}\big(\R^d\backslash (\{|x|< K^{1/a}\})\times (0,u]\big),$$
and hence
\begin{align*} \P\set{(M_0^\infty)^a \geq K } &= 
\int_0^1 1-\exp\big(-\lambda_{K^{1/a},u}\big) \, du \asymp \int_0^1 \lambda_{K^{1/a},u} \, du \asymp \lambda([K^{1/a},\infty)),\end{align*}
recalling the asymptotic edge length distribution $\lambda$ defined in \eqref{edgeLengthLimitMeasure}. The established tail behaviour of the measure
$\lambda$  yields $\P\set{(M_0^\infty)^a \geq K}\asymp 1 \wedge K^{-\eta/a}$.
\end{proof}
\pagebreak[3]

Using this, we can establish a result about the average rescaled length in the network~$G_t.$

\pagebreak[3]

\begin{theorem} \label{Theorem:ExpectedEdgeLength}
Suppose that there exists $\delta > 1$ such that the profile function satisfies $\varphi(x^d) \asymp 1 \wedge x^{-d\delta}$. Then, for all $a>0$ and $b\in [0,\frac{\eta}{a})$, there exists a posititive constant $C$, depending on $a,b,\gamma,\beta, \varphi$, such that
\begin{align}
	\frac{1}{|E(G_t)|}\sum\limits_{\x\in G_t}\bigg(\sum\limits_{\heap{\y\in G_t}{\x\overset{}{\leftrightarrow}\y}} \left(t^{1/d}d(\x,\y)\right)^a\bigg)^b \to C \label{ExpectedEdgeLengthStatement}
\end{align}
in probability, as $t\to \infty$. 
\end{theorem}

{\bf Remark:} If $\eta>1$ one can choose $a=b=1$ and this yields that the mean edge length in $G_t$ is of order $t^{-1/d}$. If $\eta<1$ (and in particular always if  $d=1$) the mean edge length is of larger order.
\medskip

\begin{proof}
Consider again a finite graph $G$ with vertices positioned in $\R^d$ and marked by birth times and with a root vertex $\x$ placed at the origin. Define $$h(G):=\Big(\sum_{\y\in\Y_\x(G)}\vert y\vert^a\Big)^b$$ 
and observe that the law of the left-hand side in \eqref{ExpectedEdgeLengthStatement} equals the law of
	\[\frac{1}{\abs{E(G^t)}}\sum_{\x \in \X^t}h(\theta_\x G^t).\]
It suffies to show that $h\in\mathcal{H}_p$ for some $p>1$, since Theorem~\ref{PropADRCM}~(iii) then ensures the convergence in probability to \smash{$\frac{1-\gamma}{\beta}\E[h(G_0^\infty)]$}, which is a positive constant. To this end recall $M_0^\infty$, the length of the longest outgoing edge of the root $\x$ in $G_0^\infty$ and observe that, almost surely, $h(G_0^t)\leq (M_0^\infty)^{a b}\vert\Y_{\x}^\infty\vert^b$. Since, by choice, $a b<\eta$, there exist some $p,q>1$, such that $\alpha:=p q a b <\eta$. Lemma~\ref{Lemma:ExpectedMaximalEdgeLength} then ensures $\E[(M_0^\infty)^\alpha]<\infty$ and, by applying H\"o{}lder's inequality to the observed bound for $h(G_0^t)$, we get
	\begin{align*}
		\sup_{t>0} \E\left[h(G_0^t)^p\right]\leq \left(\E \left[(M_0^\infty)^{\alpha}\right]\right)^{1/q} \left(\E \left[\vert \Y_\x^\infty\vert^{\frac{\alpha}{a(q-1)}}\right]\right)^{\frac{q-1}{q}} < \infty.
	\end{align*}
\end{proof}


\section{Conclusion and Outlook}

We have seen that properties of real networks, like scale-free degree distributions and clustering,
emerge from the simple building principle of preferential attachment to old and near nodes. 
Our model simplifies the spatial preferential attachment model in the literature and this allows for more explicit calculations of asymptotic network metrics. 
Moreover, we are therefore confident that we can also cover more complicated features that have proved elusive in the full, degree-based, model. In particular, for the small-world property, robustness and vulnerability of the age-based spatial preferential attachment model only partial results have been possible \cite{JacobMoerters2017,Hirsch2018}. A full study of these problems has been initiated in our group  and we hope to be able to report on new results soon.
\smallskip

Mathematically, our research is a step in the important direction of developing methods for networks that, due to clustering, cannot be locally approximated by trees. We have seen that in our case there is still a valuable description of a local limit given in terms of a tractable graph, the age-dependent random connection model. This model is interesting in its own right and methods from the theory of random geometric graphs as well as technqiues to investigate long-range percolation models can presumably be developed and enhanced to provide a powerful toolbox for its investigation.
\smallskip

The results we have achieved (and hope to achieve soon) are to some extent universal, i.e.~other network models that follow our building principles should show qualitatively very similar behaviour. But these results do not necessarily explain the full picture, other building principles could lead to similar behaviour and add to the explanation of the abundance of networks with the mentioned features that describe complex systems. A particularly interesting way to generate networks based on universal principles is to 
define random graphs that try to optimize certain functionals, for example in the form of Gibbs-measures on graphs with Hamiltonians that reward connectivity and punish long edges, see for example recent work of Mourrat and Valesin~\cite{Mouratt2018}. There remain a lot of interesting  challenges for probabilists in the area of random networks.
\bigskip

\pagebreak[3]

\appendix
\section{Simulation of the model}

In this section, we give an overview of the code used to generate the pictures shown throughout the paper. It is also used for estimating the limiting average clustering coefficient in Chapter $4$. The code can be freely accessed at: \url{http://www.mi.uni-koeln.de/~moerters/LoadPapers/adrc-model.R}.

The main objective of the code is to sample neighbours of a given vertex $(x,u)$ in the age-dependent random connection model in dimension 1 for given parameters $\beta$ and $\gamma$ and the profile function $\varphi$. 
Due to Proposition \ref{Lemma:InOutprocess}, which gives an explicit description of the neighbourhood of a given vertex, we can use rejection sampling to achieve this.
Since the support of the profile function $\varphi$ on $\mathbb{R} \times (0,1]$ may be unbounded and $\varphi$ is heavy-tailed in the second parameter, we restrict the sampling to a region with mass $q = 0.99$ with respect to $\varphi$.
This sampling works for arbitrary but reasonable choices of the profile function $\varphi$ and parameters $\beta$, $\gamma$; we provide and use an optimized sampling algorithm for $\varphi=\frac{1}{2a}\mathbbm{1}_{[0,a]}$ with $a \geq \frac{1}{2}$. The advantage of studying this class of  $\varphi$ is that expressions can be analytically simplified, which allows us to improve the algorithm by dividing the region from which the points are sampled into sub-regions with equal mass with respect to $\varphi$, thus increasing the acceptance rate for points sampled far away from $(x,u)$. That is, the code first selects one of these equally likely sub-regions uniformly at random and then points are sampled therein until one is accepted. The numerical optimization method \texttt{nlminb} is used to calculate the boundaries of the ranges, i.e. quantiles of the mass with respect to $\varphi$.

A first application of the sampling is the estimation of the expected local clustering coefficient of a vertex $(0,u)$ in the age-dependent random connection model (see Figure \ref{Img:Localclustering_varyingu}) and by Theorem \ref{Theorem:Clustering} also the average clustering coefficient for the age-based preferential attachment network (see Figure \ref{Img:Localclustering_randomU}). To this end, the code samples pairs of neighbours of $(0,u)$ and averages the probability that the pair is connected.
A second application of the sampling is generating heatmaps of the neighborhoods of a given vertex (see Figure \ref{Figure:Neighbourhood}). The heatmaps are generated using the R library \texttt{MASS} and function \texttt{kde2d} by estimating the heat kernel for the sampled neighbouring vertices. Further properties thereof can be studied with additional heatmap generating functions that we provide.

\vspace{0.5cm}

{\bf Acknowledgments:} We would like to thank Sergey Foss for the invitation to the 
\emph{Stochastic Networks 2018} workshop at ICMS, Edinburgh, where this paper was first presented.

\bigskip

\bibliographystyle{abbrv}

\bibliography{agedependent}

\end{document}